\documentclass{amsart}

\usepackage{geometry}
\usepackage[english]{babel}
\usepackage{latexsym, mathrsfs}
\usepackage{amsfonts, amsthm, amsmath, amssymb}
\usepackage{graphicx}
\usepackage[table]{xcolor}
\usepackage{amsaddr}

\newtheorem{prop}{Proposition}[section]
\newtheorem{lem}[prop]{Lemma}
\newtheorem{cor}[prop]{Corollary}
\newtheorem{thm}[prop]{Theorem}
\newtheorem{conj}[prop]{Conjecture}

\theoremstyle{definition}

\newtheorem{defi}[prop]{Definition}
\newtheorem{ex}[prop]{Example}

\numberwithin{equation}{section}

\def\Z{\mathbb{Z}}

\def\omb{\cellcolor{black!5}}

\def\H{\mathrm{H}}
\def\SMA{\mathrm{SMA}}
\def\MR{\mathrm{MR}}
\def\supp{\mathsf{supp}}
\def\skel{\mathsf{skel}}
\def\lcm{\mathrm{lcm}}

\def\B{\mathcal{B}}
\def\E{\mathcal{E}}

\begin{document}

\title{Rectangular Heffter arrays: a reduction theorem}

\author[Fiorenza Morini]{Fiorenza Morini}
\address{Dipartimento di Scienze Matematiche, Fisiche e Informatiche, Universit\`a di Parma,\\
Parco Area delle Scienze 53/A, 43124 Parma, Italy}
\email{fiorenza.morini@unipr.it}

\author[Marco Antonio Pellegrini]{Marco Antonio Pellegrini}
\address{Dipartimento di Matematica e Fisica, Universit\`a Cattolica del Sacro Cuore,\\
Via della Garzetta 48, 25133 Brescia, Italy}

\email{marcoantonio.pellegrini@unicatt.it}

\begin{abstract}
Let $m,n,s,k$ be four integers such that $3\leq s \leq n$, $3\leq k\leq m$ and $ms=nk$.
Set $d=\gcd(s,k)$. In this paper we show how one can  construct 
a Heffter array $\H(m,n;s,k)$ starting from a square Heffter array $\H(nk/d;d)$ whose elements belong to $d$ 
consecutive diagonals. As an example of application of this method, we prove that
there exists an integer $\H(m,n;s,k)$ in each of the following cases:
$(i)$ $d\equiv 0 \pmod 4$; 
$(ii)$  $5\leq d\equiv 1 \pmod 4$ and $n k\equiv 3\pmod 4$;
$(iii)$ $d\equiv 2 \pmod 4$ and $nk\equiv 0 \pmod 4$; 
$(iv)$ $d\equiv 3 \pmod 4$ and $n k\equiv 0,3\pmod 4$.
The same method can be applied also for  signed magic arrays $\SMA(m,n;s,k)$ and for magic rectangles $\MR(m,n;s,k)$.
In fact, we prove that there  exists an $\SMA(m,n;s,k)$ when $d\geq 2$, and
there exists an $\MR(m,n;s,k)$ when  either $d\geq 2$ is even or $d\geq 3$ and $nk$  are odd.
We also provide constructions of integer Heffter arrays and signed magic arrays when $k$ is odd and $s\equiv 0 \pmod 4$.
\end{abstract}


\maketitle

\section{Introduction}
 
Heffter arrays are partially filled (pf, for short) arrays introduced by Archdeacon in \cite{A}.

\begin{defi}\label{def:H}
A \emph{Heffter array} $\H(m,n; s,k)$ is an $m \times n$ pf array with elements in $(\Z_{2nk+1},+)$ such that
\begin{itemize}
\item[(\rm{a})] each row contains $s$ filled cells and each column contains $k$ filled cells;
\item[(\rm{b})] for every $x\in \Z_{2nk+1}\setminus\{0\}$, either $x$ or $-x$ appears in the array;
\item[(\rm{c})] the elements in every row and column sum to $0$ (in $\Z_{2nk+1}$).
\end{itemize}
\end{defi}

Trivial necessary conditions for the existence of an $\H(m,n; s,k)$ are $ms=nk$, $3\leq s \leq n$ and $3\leq k \leq m$.
Instead of working with elements of a finite cyclic group, one can work with integers.

\begin{defi}\label{def:HZ}
An \emph{integer} Heffter array $\H(m,n; s,k)$ is an $m \times n$ pf array with elements in $\{\pm 1, \pm 2, \ldots, \pm nk\}\subset \Z$
such that
\begin{itemize}
\item[(\rm{a})] each row contains $s$ filled cells and each column contains $k$ filled cells;
\item[(\rm{b})] no two entries agree in absolute value;
\item[(\rm{c})] the elements in every row and column sum to $0$.
\end{itemize}
\end{defi}

As shown in \cite{A}, an additional necessary condition for the existence of an integer $\H(m,n;s,k)$ is that
$nk\equiv 0,3 \pmod 4$. 
The study of these objects began with the square case 
(i.e., when $m=n$, and so $s=k$), and the tight case (i.e., when $m=k$ and $n=s$). 
In particular,  in \cite{CDDY} it was proved that a square Heffter array $\H(n,n;k,k)$ (that we simply denote by
$\H(n;k)$) exists for all $n\geq k\geq 3$, while by \cite{ADDY, DW} an integer Heffter array $\H(n;k)$ exists if and 
only if $n\geq k\geq 3$ and $nk\equiv 0,3\pmod 4$.
On the other hand, by  \cite{ABD} an $\H(m,n;n,m)$ exists for all $m,n\geq 3$, while an integer $\H(m,n;n,m)$ exists if and only 
if the additional condition $mn\equiv 0,3 \pmod 4$ holds.
These results confirm the validity of the following conjecture, originally proposed by Archdeacon himself.

\begin{conj}\cite[Conjecture 6.3]{A}\label{ConjArch}
Given four integers $m,n,s,k$ such that $3\leq s\leq n$, $3\leq k \leq m$ and $ms=nk$,  there exists a Heffter array $\H(m,n;s,k)$.
If the additional condition $nk\equiv 0,3 \pmod 4$ holds, there exists an integer Heffter array $\H(m,n;s,k)$.
\end{conj}

The first aim of this paper is to show how it is possible to reduce the problem of the  existence of a rectangular Heffter array 
to the square case where the elements belong to consecutive diagonals (see Theorem \ref{sq->rt}).
This process, combined with \cite[Theorem 1.3]{MP}, allows us to prove the following.

\begin{thm}\label{main}
Let $m,n,s,k$ be four integers such that $3\leq s \leq n$, $3\leq k\leq m$ and $ms=nk$.
Set $d=\gcd(s,k)$. There exists an integer $\H(m,n;s,k)$ in each of the following cases:
\begin{itemize}
\item[(1)] $d\equiv 0 \pmod 4$;
\item[(2)] $d\equiv 1 \pmod 4$ with $d\geq 5$ and $n k\equiv 3\pmod 4$;
\item[(3)] $d\equiv 2 \pmod 4$ and $nk\equiv 0 \pmod 4$; 
\item[(4)] $d\equiv 3 \pmod 4$ and $n k\equiv 0,3\pmod 4$.
\end{itemize}
\end{thm}

Thanks to this result, Conjecture \ref{ConjArch}  for the integer case remains open only when $d=1$
or $5\leq d\equiv 1 \pmod 4$ and $nk\equiv 0 \pmod 4$.
In particular, one can assume $k$ odd: in fact, the transpose of an (integer) $\H(m,n;s,k)$ is an (integer)
$\H(n,m;k,s)$. A further result in this direction is given in Section \ref{s0}, where
we construct integer Heffter arrays  when $s\equiv 0\pmod 4$.

\begin{thm}\label{main_odd}
Let  $m,n,s,k$ be four integers such that $3\leq s\leq n$, $3\leq k \leq m$ and $ms=nk$.
If $s\equiv 0 \pmod 4$ and $k\neq 5$ is odd, then  there exists an integer Heffter array $\H(m,n;s,k)$.
\end{thm}

Our reduction theorem will be proved in the more general context of $\lambda$-fold relative Heffter arrays,
objects introduced in \cite{SA}.
One of the main reasons that justifies the study of Heffter arrays and of their generalizations is that they allow,
under suitable conditions (or assuming the validity of \cite[Conjecture 3]{CMPP}), to produce biembeddings of 
pairs of cyclic decompositions (one decomposition consisting of $s$-cycles and the other one consisting of $k$-cycles) 
of the complete multipartite  multigraph ${}^\lambda K_{\ell\times t}$, with $\ell$ parts each of size $t$,
onto orientable surfaces (see \cite{BCDY,CDYBiem,CMPPHeffter,SA,RelHBiem,DM}).
Partial results about the existence of these generalizations have been obtained in \cite{RelH,SA,RelHBiem,MP,MP2}.

In Section \ref{sec:magic} we provide a similar reduction theorem for magic rectangles and signed magic arrays.

\begin{defi}\label{def:rect}
A \emph{magic rectangle} $\MR(m,n; s,k)$ is an $m\times n$ pf  array
with elements in $\Omega=\{0,1,\ldots,nk-1\}\subset \Z$ such that
\begin{itemize}
\item[{\rm (a)}] each row contains $s$ filled cells and each column contains $k$ filled cells;
\item[{\rm (b)}] every $x \in \Omega$ appears exactly once in the array;
\item[{\rm (c)}] the sum of the elements in each row is a constant value $c_1$  and
the sum of the elements in each column is a constant value $c_2$.
\end{itemize}
\end{defi}

\begin{defi}\label{def:Magic}
A \emph{signed magic array} $\SMA(m,n; s,k)$ is an $m\times n$ pf  array
with elements in $\Omega\subset \Z$, where $\Omega=\{0,\pm 1, \pm 2,\ldots, \pm (nk-1)/2\}$ if $nk$ is odd and
$\Omega=\{\pm 1, \pm 2, \ldots, \pm nk/2\}$ if $nk$ is even, such that
\begin{itemize}
\item[{\rm (a)}] each row contains $s$ filled cells and each column contains $k$ filled cells;
\item[{\rm (b)}] every $x \in \Omega$ appears exactly once in the array;
\item[{\rm (c)}] the elements in every row and column sum to $0$.
\end{itemize}
\end{defi}

The existence problem for diagonal magic squares (that is, magic squares whose elements belong to consecutive diagonals) has been already solved in \cite{rect1}: 
then our reduction theorem gives the following.

\begin{thm}\label{magic_rect}
Let  $m,n,s,k$ be four integers such that $3\leq s\leq n$, $3\leq k \leq m$ and $ms=nk$.
Set $d=\gcd(s,k)$ and suppose $d\geq 2$. 
If $d$ is even or $nk$ is odd, then there exists an $\MR(m,n;s,k)$.
\end{thm}

In Section \ref{sec:magic} we also consider the existence of a diagonal square $\SMA(n;k)$.
In fact, some of the signed magic squares obtained by Khodkar, Schulz, and Wagner in  \cite{magicDM} are not diagonal. 
So, we firstly construct a diagonal $\SMA(n;3)$ for all $n\geq 3$ such that $n\not \equiv 0\pmod 4$,
a diagonal $\SMA(n;5)$ for all odd $n\geq 5$, and a diagonal $\SMA(n;6)$ for all $n\geq 6$.
This allows us to prove that a diagonal $\SMA(n;k)$ exists for all $n\geq k \geq 3$.
Then, applying our reduction theorem we obtain the following result.

\begin{thm}\label{magic}
Let  $m,n,s,k$ be four integers such that $3\leq s\leq n$, $3\leq k \leq m$ and $ms=nk$.
There  exists an $\SMA(m,n;s,k)$ whenever $\gcd(s,k)\geq 2$. 
\end{thm}

To conclude our paper, we construct signed magic arrays when $s \equiv 0\pmod 4$.

\begin{thm}\label{magic4}
Let $m,n,s,k$ be four integers such that $3\leq s\leq n$, $3\leq k \leq m$ and $ms=nk$.
If $s\equiv 0 \pmod 4$, then  there exists an $\SMA(m,n;s,k)$.
\end{thm}

\section{Notation}

In this paper,  the arithmetic on the row (respectively, on the column) indices is performed modulo $m$
(respectively, modulo $n$), where the set of reduced residues is $\{1,2,\ldots,m\}$ (respectively,
 $\{1,2,\ldots,n\}$), while the entries of the arrays are taken in $\Z$.
Given two integers $a\leq b$, we denote by $[a,b]$ the interval consisting of the integers $a,a+1,\ldots,b$.
If $a>b$, then $[a,b]$ is empty. 

We denote by $(i,j)$ the cell in the $i$-th row and $j$-th column 
of a pf array $A$. The skeleton of $A$ is the set of its filled positions: it will be denoted by $\skel(A)$.
The \emph{support} of  $A$, denoted by $\supp(A)$, is defined to be the set of the absolute
values of the elements contained in $A$, while $\E(A)$ denotes the \emph{list} of the entries of the filled cells of $A$.
We also write $\E(i,j)$ to indicate the entry of the cell $(i,j)$ of $A$.
Given a sequence $S=(B_1,B_2,\ldots,B_r)$ of pf arrays,
we set $\supp(S)=\cup_i \supp(B_i)$ and $\E(S) = \cup_i \E(B_i)$.

Let $A=(a_{i,j})$ be a pf square array of size $n$. We say that the
element $a_{i,j}$ belongs to the 
diagonal $D_{r}$ if $j-i\equiv r \pmod{n}$.
Moreover, $A$ is said to be cyclically $b$-diagonal if
the nonempty cells of $A$ are exactly those of $b$ consecutive diagonals.
For instance, the pf array $A$ given in  Figure \ref{fig2} is an integer $\H(12;3)$ which is 
cyclically $3$-diagonal. In fact, $\skel(A)=D_{0}\cup D_{1}\cup D_{2}$.

\begin{figure}
\begin{footnotesize}
 $$\begin{array}{|c|c|c|c|c|c|c|c|c|c|c|c|} \hline
 -36 & 12 & 24 & &  &  &  &  &  &  &  &  \\\hline
 & -25 & -10& 35& &  &  &  &  &  &  &     \\ \hline
 & &  -14& -9& 23& &  &  &  &  &  &     \\\hline
 & &  &  -26& -8& 34& &  &  &  &  &     \\ \hline
 & &  &  &  -15& -7& 22& &  &  &  &     \\ \hline
 & &  &  &  &  -27& 11& 16& &  &  &     \\ \hline
 & &  &  &  &  &  -33& 5& 28& &  &    \\\hline
 & &  &  &  &  &  &  -21& 4& 17& &     \\ \hline
 & &  &  &  &  &  &  &  -32& 3& 29&    \\ \hline
 & &  &  &  &  &  &  &  &  -20& 2& 18  \\ \hline
30 &  &  &  &  &  &  &  &  &  &  -31 & 1 \\ \hline
 6 & 13 & &  &  &  &  &  &  &  &  &  -19  \\   \hline
  \end{array}$$
  \end{footnotesize}
\caption{An integer cyclically $3$-diagonal $\H(12;3)$.}\label{fig2}
\end{figure}

\begin{defi}
A pf array with entries in $\Z$ is said to be \emph{shiftable} if
every row and every column contains an equal number of positive and negative entries.
\end{defi}

Let $A$ be a shiftable pf array and $x$ be a nonnegative integer.
Let $A\pm x$ be the (shiftable) pf array obtained by adding $x$ to each positive
entry of $A$ and $-x$   to each negative entry of $A$. 
Observe that, since  $A$ is shiftable, the row and column sums of $A\pm x$ are exactly
the row and column sums of $A$.

\section{The reduction theorem}

In this section we show how one can construct rectangular Heffter arrays starting from a square Heffter array having a particular shape.
This result will be proved in the more general context of $\lambda$-fold relative Heffter arrays, see \cite{SA,MP2}.

\begin{defi}\label{LambdaH}
Let $m,n,s,k,t,\lambda$ be positive integers such that $\lambda$ divides $2nk$ and $t$ divides $\frac{2nk}{\lambda}$.
Let $J$ be the subgroup of order $t$ of $\Z_v$, where $v=\frac{2nk}{\lambda}+t$.
A \emph{$\lambda$-fold Heffter array} over $\Z_v$ relative to $J$, denoted by ${}^\lambda \H_t(m,n;s,k)$, is
an $m\times n$ pf array $A$ with elements in $\Omega=\Z_v\setminus J$ such that:
\begin{itemize}
\item[{\rm (a)}] each row contains $s$ filled cells and each column contains $k$ filled cells;
\item[{\rm (b)}] every element of $\Omega$ appears exactly $\lambda$ times in the list $\E(A)\cup -\E(A)$;
\item[{\rm (c)}] the elements in every row and column sum to $0$.
\end{itemize}
\end{defi}

Also for these arrays, there exists an `integer' version.

\begin{defi}\label{LambdaHInt}
Let $m,n,s,k,t,\lambda$ be positive integers such that $\lambda$ divides $2nk$ and $t$ divides $\frac{2nk}{\lambda}$.
Let 
$$\Phi=\left\{1,2,\ldots,\left\lfloor\frac{v}{2} \right\rfloor   \right\}\setminus \left\{\ell, 2\ell,\ldots,
\left\lfloor \frac{t}{2}\right\rfloor \ell \right \} \subset \Z,\quad \textrm{where } 
v=\frac{2nk}{\lambda}+t \textrm{ and } \ell=\frac{v}{t}.$$ 
An \emph{integer} ${}^\lambda \H_t(m,n;s,k)$ is
an $m\times n$ pf array with elements in $\Phi$ such that:
\begin{itemize}
\item[{\rm(a)}] each row contains $s$ filled cells and each column contains $k$ filled cells;
\item[{\rm(b)}] if  $v$ is odd or if  $t$ is even,  every element of $\Phi$ appears, up to sign, exactly 
$\lambda$ times in the array; if $v$ is even and  $t$ is odd, every element of $\Phi\setminus \{\frac{v}{2}\}$
appears, up to sign, exactly 
$\lambda$ times while $\frac{v}{2}$ appears, up to sign, exactly $\frac{\lambda}{2}$ times;
\item[{\rm(c)}] the elements in every row and column sum to $0$.
\end{itemize}
\end{defi}

Clearly, an (integer) ${}^1 \H_1(m,n;s,k)$ is nothing but an (integer) $\H(m,n;s,k)$. In the following, we write
${}^\lambda \H_t(n;k)$ for ${}^\lambda \H_t(n,n;k,k)$.
Also, we say that a ${}^\lambda \H_t(n;k)$  is diagonal if it is cyclically $k$-diagonal.

\begin{thm}\label{sq->rt}
Let $m,n,s,k$ be four integers such that $2\leq s \leq n$, $2\leq k\leq m$ and $ms=nk$.
Let $\lambda$ be a divisor of $2nk$ and let $t$ be a divisor of $\frac{2nk}{\lambda}$.
Set $d=\gcd(s,k)$.
If  there exists a (shiftable/integer) diagonal ${}^\lambda \H_t\left( \frac{nk}{d}; d \right)$,
then there exists a (shiftable/integer) ${}^\lambda \H_t( m,n; s,k)$.
\end{thm}

\begin{proof}
Write $s=d\bar s$ and $k=d \bar k$. Since $\gcd(\bar s, \bar k)=1$, the equality $ms=nk$  implies that $\bar k$ divides $m$
and $\bar s$ divides $n$.
Hence, we can write $m=\bar k c$ and  $n=\bar s c$. Note that $c=\gcd(m,n)$.
Let $A=(a_{i,j})$ be an (integer) diagonal ${}^\lambda \H_t(n\bar k; d )$
and suppose that the filled cells of $A$ belongs to the diagonals $D_0,D_1,\ldots,D_{d-1}$, 
where 
$$D_i=\{(x,y): 1\leq x,y\leq n\bar k,\; y-x\equiv i \pmod{n\bar k}\}.$$
We define a function $\psi$ between $\skel(A)$ and the set of the cells of  an empty array $B$ of size $m\times n$, as follows:
for every  $(i,j)\in \skel(A)$ with $i,j \in [1, n\bar k ]$, let $\psi((i,j))=(u,v)$,
where $u\in [1,m]$ and $v \in [1,n]$ are such that $u \equiv i \pmod m$ and $v \equiv j \pmod n$.

We prove that the function $\psi$ is injective.
In fact, take two filled cells $(x_1,y_1)\in D_{i_1}$ and $(x_2,y_2)\in D_{i_2}$ of $A$, where $0\leq i_1\leq i_2< d$.
Suppose $\psi((x_1,y_1))=\psi((x_2,y_2))$. Then  $x_2\equiv x_1 \pmod m$ and $y_2\equiv y_1\pmod n$ and hence $y_2-x_2 \equiv y_1-x_1 \pmod c$.
This implies  $i_2\equiv i_1 \pmod c$ and so
either $i_2=i_1$ or $i_2-i_1\geq  c=\frac{nd}{s}\geq d$, since $n\geq s$.
As the second possibility is excluded by the hypothesis $i_2-i_1< d$, the two cells  $(x_1,y_1),(x_2,y_2)$ belong
to the same diagonal $D_{i}$. 
From $y_2-x_2\equiv y_1-x_1\pmod{n\bar k}$, we get $x_2\equiv x_1 \pmod n$. Since $x_2\equiv x_1 \pmod m$, we conclude
that $x_2\equiv x_1 \pmod {\lcm(m,n)}$. Thus, from $\lcm(m,n)=\bar s c\bar k=n\bar k$ we get $(x_1,y_1)=(x_2,y_2)$.

Because of the injectivity of $\psi$, we can fill the cell $\psi((i,j))$ of $B$ with the entry $a_{i,j}$ of $A$,
obtaining a bijection from $\skel(A)$ to $\skel(B)$.
We now prove  that the pf array $B$ so constructed is an (integer) ${}^\lambda \H_t( m,n; s,k)$.
Clearly, $\E(A)=\E(B)$.  Also, each row of $B$ contains $d\cdot \frac{m\bar s}{m}=s$ filled cells and 
every column of $B$ contains $d\cdot \frac{n\bar k}{n}=k$ filled cells.
Finally, take the filled cells of a row of $A$: $(x,y_1), (x,y_2),\ldots, (x, y_{d})$.
Applying the function $\psi$, we obtain cells that belong to the same row $u$ of $B$.
So, the elements of this row  of $B$ are exactly the elements of $\bar s$ distinct rows of $A$.
Since the elements of any row of $A$ sum to zero in $\Z_{\frac{2nk}{\lambda}+t}$ (or in $\Z$), the corresponding elements of the 
row $u$ of $B$ sum to zero in $\Z_{\frac{2nk}{\lambda}+t}$ (respectively, in $\Z$). The same for the columns.

This shows that $B$ is an (integer) ${}^\lambda \H_t( m,n; s,k)$. 
If we make the further assumption that $A$ is a shiftable diagonal ${}^\lambda \H_t(n\bar k; d )$, 
then $d$ is even and every row/column of $A$ has $\frac{d}{2}$ positive entries and $\frac{d}{2}$ negative entries.
As remarked before, the elements of every row of $B$ are exactly the elements of $\bar s$ distinct rows of $A$. Hence, each row of $B$
has $\frac{s}{2}$ positive entries and $\frac{s}{2}$ negative entries. Similarly,
each column of $B$ has $\frac{k}{2}$ positive entries and $\frac{k}{2}$ negative entries, proving that $B$ is a shiftable
${}^\lambda \H_t( m,n; s,k)$. 
\end{proof}
 
For instance, we can take the integer diagonal $\H(12;3)$ of Figure \ref{fig2}, obtained by using \cite[Theorem 3.11]{ADDY}.
Following the proof of Theorem \ref{sq->rt}, we obtain the integer $\H(6,12; 6,3)$ of Figure \ref{fig6}.

\begin{figure}[ht]
\begin{footnotesize}
$$\begin{array}{|c|c|c|c|c|c|c|c|c|c|c|c|} \hline
-36 & 12 & 24 &  &  &  & -33 & 5 & 28 &  &  &   \\   \hline
 & -25 & -10 & 35 &  &  &  & -21 & 4 & 17 &  &  \\   \hline
&  & -14 & -9 & 23 &  &  &  & -32 & 3 & 29 &    \\   \hline
&  &  & -26 & -8 & 34 &  &  &  & -20 & 2 & 18   \\   \hline
30 &  &  &  & -15 & -7 & 22 &  &  &  & -31 & 1  \\   \hline
6 & 13 &  &  &  & -27 & 11 & 16 &  &  &  & -19  \\   \hline
  \end{array}$$
\end{footnotesize}
\caption{An integer $\H(6,12; 6,3)$.}\label{fig6}
\end{figure}
As an example of application of our reduction theorem, we get the following.

\begin{cor}\label{cor:main}
Let $m,n,s,k$ be four integers such that $3\leq s \leq n$, $3\leq k\leq m$ and $ms=nk$.
Suppose $d=\gcd(s,k)\geq 3$.
\begin{itemize}
\item[(1)] If $d\equiv 0 \pmod 4$, then there exists a shiftable $\H(m,n;s,k)$.
\item[(2)] If $d\equiv 1 \pmod 4$ and $n k\equiv 3\pmod 4$, then there exists an integer $\H(m,n;s,k)$.
\item[(3)] If $d\equiv 3 \pmod 4$ and $n k\equiv 0,3\pmod 4$, then there exists an integer $\H(m,n;s,k)$.
\end{itemize}
\end{cor}

\begin{proof}
We recall that there exists an integer diagonal $\H(a;b)$, with $a\geq b\geq 3$, in each of the following cases:
\begin{itemize}
\item[(i)] $b\equiv 0\pmod 4$ (shiftable), see \cite{ADDY} and \cite{MP};
\item[(ii)] $b\equiv 1\pmod 4$ and $a\equiv 3\pmod 4$, see \cite{CMPPHeffter,DW}.
\item[(iii)] $b\equiv 3\pmod 4$ and $a\equiv 0,1\pmod 4$, see \cite{ADDY}.
\end{itemize}
Now, it suffices to apply Theorem \ref{sq->rt} taking $b=d$ and $a=\frac{nk}{d}$
(note that $\frac{nk}{d}\geq d$ as $n\geq s\geq d$).
\end{proof}

\begin{ex}
We  construct the following integer $\H(8,12; 9,6)$ starting from an integer diagonal $\H(24;3)$:
$$\begin{array}{|c|c|c|c|c|c|c|c|c|c|c|c|}\hline
-72& 24 & 48 &  & -64 & 7 & 57 &  & -29 & -15 & 44 &   \\ \hline
& -49 & -22 & 71 &  & -40 & 6 & 34 &  & -53 & -14 & 67  \\ \hline 
43 &  & -26 & -21 & 47 &  & -63 & 5 & 58 &  & -30 & -13  \\ \hline
23 & 31 &  & -50 & -20 & 70 &  & -39 & 4 & 35 &  & -54  \\ \hline
-66 & 11 & 55 &  & -27 & -19 & 46 &  & -62 & 3 & 59 &   \\ \hline
 & -42 & 10 & 32 &  & -51 & -18 & 69 &  & -38 & 2 & 36   \\ \hline
60 &  & -65 & 9 & 56 &  & -28 & -17 & 45 &  & -61 & 1  \\ \hline
12 & 25 &  & -41 & 8 & 33 &  & -52 & -16 & 68 &  & -37  \\ \hline
  \end{array}$$
\end{ex}

\begin{proof}[Proof of Theorem \ref{main}]
By   \cite[Theorem 1.3]{MP}, we may assume that $d\geq 3$ is odd.
So, the result follows from Corollary \ref{cor:main}, items (2) and (3).
\end{proof}

We now apply our reduction theorem to the known results on diagonal $\lambda$-fold relative Heffter arrays.

\begin{cor}
Let $m,n,s,k$ be four integers such that $3\leq s \leq n$, $3\leq k\leq m$ and $ms=nk$.
Suppose $d=\gcd(s,k)\geq 3$.
\begin{itemize}
\item[(1)] If $d\equiv 1 \pmod 4$ and $n k\equiv 3\pmod 4$, then there exists an integer ${}^1 \H_d(m,n;s,k)$.
\item[(2)] If $d\equiv 3 \pmod 4$ and $n k\equiv 0,1\pmod 4$, then there exists an integer ${}^1 \H_d(m,n;s,k)$.
\item[(3)] If $d=3$ and $nk$ is odd, then there exists an integer ${}^1\H_\frac{nk}{3}(m,n;s,k)$ and an integer
${}^1\H_{\frac{2nk}{3}}(m,n;s,k)$.
\item[(4)] If $d=3$ and $nk\equiv 3 \pmod 4$, then there exists a (non-integer)  ${}^2 \H_1 (m,n;s,k)$.
\item[(5)] If $d=3$ and $nk\equiv 1 \pmod 4$, then there exists a  (non-integer) ${}^3 \H_1 (m,n;s,k)$.
\item[(6)] If $d=3$, $nk$ is odd, and $\lambda$ divides $n$, then there exists a  (non-integer) ${}^\lambda \H_{\frac{n}{\lambda}}(m,n;s,k)$.
\item[(7)] If $d=3$, $nk$ is odd, and $\lambda$ divides $2n$, then there exists a  (non-integer) ${}^\lambda \H_{\frac{2n}{\lambda }}(m,n;s,k)$.
\item[(8)] If $d=5$ and $nk\equiv 3 \pmod 4$, then there exists a (non-integer)  ${}^2 \H_1 (m,n;s,k)$.
\end{itemize}
\end{cor}

\begin{proof}
There exists an integer diagonal ${}^1 \H_b(a;b)$, with $a\geq b\geq 3$,   when 
$b\equiv 1\pmod 4$ and $a\equiv 3\pmod 4$, and when $b\equiv 3\pmod 4$ and $a\equiv 0,3\pmod 4$, see  \cite{RelH}.
By \cite{RelHBiem},  there exist an integer diagonal
${}^1\H_a(a;3)$ and an integer diagonal ${}^1\H_{2a}(a;3)$  when $a\geq 3$ is odd.
By \cite{SA}, there exist the following diagonal non-integer arrays:
\begin{itemize}
 \item[(i)]  ${}^2 \H_1(a;3)$ when $a\geq 5$ is such that $a\equiv 1 \pmod 4$;
 \item[(ii)] ${}^3 \H_1(a;3)$ when $a\geq 3$ is such that $a\equiv 3 \pmod 4$,
\item[(iii)] ${}^\lambda \H_{\frac{n}{\lambda}}(a;3)$ when $a\geq 3$ is odd and  $\lambda$ divides $n$;
\item[(iv)]  ${}^\lambda \H_{\frac{2n}{\lambda}}(a;3)$ when  $a\geq 3$ is odd and  $\lambda$ divides $2n$;
 \item[(v)]  ${}^2 \H_1(a;5)$ when $a\geq 7$ is such that $a\equiv 3\pmod 4$. 
\end{itemize}
Now, it suffices to apply Theorem \ref{sq->rt} taking $b=d$ and $a=\frac{nk}{d}$.
\end{proof}

\section{Direct constructions when $k$ is odd}\label{s0}

In this section, we construct integer Heffter arrays $\H(m,n;s,k)$ when $k\neq 5$ is odd and $s\equiv 0 \pmod 4$.
Note that from the necessary condition $ms=nk$ we obtain $n\equiv 0\pmod 4$.
By Theorem \ref{main} we could also assume $\gcd(s,k)=1$, however this hypothesis is not necessary for our constructions.

We begin considering the case $k=3$. In \cite{ABD} the authors proved the existence of an integer $\H(n,3;3,n)$ using blocks of size 
$3\times 8$ and blocks of size $3\times 12$. 
Our first step is to rearrange the elements of these blocks, obtained by working with Skolem sequences \cite{Skol}, in order
to produce $3\times 4$ blocks whose rows and columns sum to zero.
So, for any fixed integer $\mu\geq 0$, we can define:
$$\begin{array}{rcl}
A_1 & =& \begin{array}{|c|c|c|c|}\hline
4\mu+4 & -8\mu-7 & 4\mu+5 & -2 \\ \hline
8\mu+9 & 18\mu+18 & -14\mu-15 & -12\mu-12 \\ \hline
-12\mu-13 & -10\mu-11 & 10\mu+10 & 12\mu+14 \\ \hline
\end{array}\;,\\[14pt]
A_2   & =& \begin{array}{|c|c|c|c|}\hline
4\mu+6 & 4\mu+3 & -1 & -8\mu-8 \\ \hline
18\mu+17  & 18\mu+19 & -20\mu-20 & -16\mu-16 \\ \hline
-22\mu-23 & -22\mu-22 & 20\mu+21 & 24\mu+24 \\ \hline
  \end{array}\;,\\[14pt]
B_a   & =& \begin{array}{|c|c|c|c|}\hline
8\mu+10+2a & 14\mu+14-2a & -8\mu-11-2a & -14\mu-13+2a \\ \hline
8\mu+5-4a & -4\mu-2+4a & -8\mu-3+4a & 4\mu-4a \\ \hline
-16\mu-15+2a & -10\mu-12-2a & 16\mu+14-2a & 10\mu+13+2a \\\hline
  \end{array}\;,\\[14pt]
C_a & =& \begin{array}{|c|c|c|c|}\hline
4\mu+1-4a & -8\mu-6+4a & -4\mu+1+4a & 8\mu+4-4a \\ \hline
18\mu+20+2a & -16\mu-17-2a & -18\mu-21-2a & 16\mu+18+2a \\ \hline
-22\mu-21+2a & 24\mu+23-2a & 22\mu+20-2a & -24\mu-22+2a \\ \hline
\end{array}\;, 
 \end{array}$$
with $a\in [0,\mu-1]$. The sequence
\begin{equation}\label{seqU1}
\mathcal{A}_1(\mu) = \left\{
\begin{array}{ll}
(A_1,A_2) & \textrm{if } \mu=0,\\
(A_1,A_2,B_0,B_1,\ldots,B_{\mu-1}, C_0, C_1,\ldots,C_{\mu-1})&  \textrm{if } \mu\geq 1
\end{array}\right.
\end{equation}
has length $2\mu+2$ and  support equals to $\supp(\mathcal{A}_1(\mu))=[1,24\mu+24]$.
In fact, 
$$\begin{array}{rcl}
\bigcup\limits_{a=1}^{2} \supp(A_a) & =& [1,2]\cup [4\mu+3,4\mu+6], \cup [8\mu+7,8\mu+9] \cup [10\mu+10,10\mu+11]\\
&& \cup [12\mu+12,12\mu+14] \cup\{14\mu+15 \} \cup \{16\mu+16\} \cup [18\mu+17,18\mu+19]  \\
&& \cup [20\mu+20, 20\mu+21] \cup [22\mu+22 , 22\mu+23] \cup \{24\mu+24\};\\
\bigcup\limits_{a=0}^{\mu-1} \supp(B_a) & = & \{4,6,\ldots,4\mu+2\} \cup \{4\mu+7,4\mu+9,\ldots,8\mu+5\}\cup [8\mu+10,10\mu+9]\\
&& \cup [10\mu+12,12\mu+11] \cup [12\mu+15,14\mu+14] \cup [14\mu+16,  16\mu+15];\\
\bigcup\limits_{a=0}^{\mu-1} \supp(C_a) & = & \{3,5,\ldots,4\mu+1\} \cup \{4\mu+8,4\mu+10,\ldots,8\mu+6\}\cup [16\mu+17,18\mu+16]\\
&& \cup[18\mu+20,20\mu+19] \cup [20\mu+22,22\mu+21] \cup [22\mu+24,  24\mu+23].
\end{array}$$
Now, for any fixed integer $\nu\geq 0$ define:
$$\begin{array}{rcl}
E_1 & =& \begin{array}{|c|c|c|c|}\hline
4\nu+1 & -4\nu-5 & -4\nu-8 & 4\nu +12   \\\hline
18\nu+30 & -18\nu-28 & -18\nu-26 & 18\nu+24 \\\hline
-22\nu-31 & 22\nu+33 & 22\nu+34 & -22\nu-36 \\\hline
\end{array}\;,\\[14pt]
E_2 & =& \begin{array}{|c|c|c|c|}\hline
4\nu+ 6 & -18\nu-25 & 4\nu+2 & 10\nu+17    \\\hline
 8\nu+13 & -4\nu-10 & 10\nu+18 & -14\nu-21 \\\hline
 -12\nu-19 & 22\nu+35 & -14\nu-20 & 4\nu+4 \\\hline
\end{array}\;,\\[14pt]
E_3 & =& \begin{array}{|c|c|c|c|}\hline
 4\nu +7 & 10\nu+14 & 8\nu+11 & -22\nu-32    \\\hline
10\nu+15 & 4\nu+9 & -18\nu-27 & 4\nu+3 \\\hline
-14\nu-22 & -14\nu-23 & 10\nu+16 & 18\nu+29 \\\hline
\end{array}\;,\\[14pt]
F_a   & =& \begin{array}{|c|c|c|c|}\hline
 4\nu-3-4a & -4\nu+2+4a  &  -4\nu+1+4a  &  4\nu-4a  \\ \hline
  18\nu+32+2a &  -10\nu-20-2a & -18\nu-31-2a & 10\nu+19+2a  \\\hline
  -22\nu-29+2a & 14\nu+18-2a &  22\nu+30-2a & -14\nu-19+2a  \\ \hline
  \end{array}\;,\\[14pt]
G_a   & =& \begin{array}{|c|c|c|c|}\hline
8\nu + 7 -4a & -8\nu-9+4a & -8\nu-10+4a & 8\nu+12-4a \\ \hline
8\nu+15+2a & -8\nu-14-2a & -16\nu-25-2a & 16\nu+24+2a \\ \hline
-16\nu-22+2a & 16\nu+23-2a & 24\nu+35-2a & -24\nu-36+2a \\ \hline
\end{array}\;, 
 \end{array}$$
with $a \in  [0, \nu-1]$. The sequence
\begin{equation}\label{seqU2}
\mathcal{A}_2(\nu) = \left\{
\begin{array}{ll}
(E_1,E_2,E_3) & \textrm{if } \nu=0,\\
(E_1,E_2,E_3,F_0,F_1,\ldots,F_{\nu-1}, G_0, G_1,\ldots,G_{\nu-1})&  \textrm{if } \nu\geq 1
\end{array}\right.
\end{equation}
has length $2\nu+3$ and  support equals to $\supp(\mathcal{A}_2(\nu))=[1,24\nu+36]$.
In fact, 
$$\begin{array}{rcl}
\bigcup\limits_{a=1}^{2} \supp(E_a) & =& [4\nu+1,4\nu+10]\cup \{4\nu+12\}\cup \{8\nu+11\}\cup \{8\nu+13\}\\
&& \cup [10\nu+14,10\nu+18] \cup \{12\nu+19\} \cup [14\nu+20,14\nu+23]\\
&& \cup [18\nu+24,18\nu+30]\cup [22\nu+31, 22\nu+36];\\
\bigcup\limits_{a=0}^{\mu-1} \supp(F_a) & = & [1,4\nu] \cup [10\nu+19, 12\nu+18] \cup[12\nu+20, 14\nu+19] \\
&&\cup [18\nu+31, 20\nu+30] \cup [20\nu+31, 22\nu+30];\\
\end{array}$$
$$\begin{array}{rcl}
\bigcup\limits_{a=0}^{\mu-1} \supp(G_a) & = &\{4\nu+11,4\nu+13,\ldots,8\nu+9\}\cup\{4\nu+14,4\nu+16,\ldots,8\nu+12\}\\
&& [8\nu+14, 10\nu+13] \cup [14\nu+24, 16\nu + 23] \cup [16\nu+24,18\nu+23]\\
&&\cup [22\nu+37, 24\nu+36].
\end{array}$$

\begin{proof}[Proof of Theorem \ref{main_odd}]
Our first step is to construct a sequence $\mathcal{Z}=(Z_0,Z_1,\ldots,Z_{\frac{n-4}{4}})$ of length $\frac{n}{4}$,
such that its elements $Z_i$ are $k\times 4$ blocks whose rows and columns sum to zero, and such that $\supp(\mathcal{Z})=[1,nk]$.

To this intent, we define a sequence $\B=(U_0,U_1,\ldots,U_{\frac{n-4}{4}})$ of length $\frac{n}{4}$ consisting of blocks of size $3\times 4$, 
as follows.
If $n\equiv 0 \pmod 8$, we take the sequence $\mathcal{B}=\mathcal{A}_1\left(\frac{n-8}{8}\right)$ as defined in \eqref{seqU1}.
If $n\geq 12$ and $n \equiv 4 \pmod 8$, we set $\mathcal{B}=\mathcal{A}_2\left(\frac{n-12}{8}\right)$ as defined in \eqref{seqU2}.
Finally, if $n=4$, the sequence $\mathcal{B}$ consists only of the block
$\H(3,4;4,3)=\begin{array}{|c|c|c|c|}\hline
1 & 2 & 3 & -6\\ \hline
8 & -12 & -7 & 11\\ \hline
-9 & 10 & 4 & -5 \\ \hline
\end{array}$.
Note that, in each case,  $\supp(\mathcal{B})=[1,3n]$.

If $k=3$, set $Z_i=U_i$ for all $i=0,\ldots,\frac{n-4}{4}$. It is clear that $\mathcal{Z}$ satisfies all our requirements.
Assume $k\geq 7$. To obtain our block $Z_i$ we use the block $U_i$ for the first three rows, and we need a shiftable $\H(k-3,4;4,k-3)$
to construct the remaining $k-3$ rows. The existence of this Heffter array was proved in \cite{ABD,MP}, but here, 
for the reader's sake, we prefer to give an explicit construction. So, consider the following shiftable blocks:
$$Q_4=\begin{array}{|c|c|c|c|}\hline
1 & -2 & -3 & 4 \\ \hline
-5 & 6 & 7 & -8 \\ \hline
-9 & 10 & 11 & -12 \\\hline
13 & -14 & -15 & 16 \\ \hline
    \end{array}\;, \quad 
    Q_6=\begin{array}{|c|c|c|c|}\hline
1 & -14 & -3 & 16 \\ \hline
-2 & 13 & 4 & -15 \\ \hline
5 & -18 & -7 & 20 \\ \hline
-6 & 17 & 8 & -19 \\ \hline
-9 & 23 & 10 & -24 \\ \hline
11 & -21 & -12 & 22 \\ \hline
\end{array}\;.$$
Note that $Q_4$ is an integer $\H(4,4;4,4)$ and $Q_6$ is an integer $\H(6,4;4,6)$.
For all fixed $h\geq 0$ we take the $(4h+4)\times 4$ block $V_h$ and the $(4h+6)\times 4$ block $W_h$, as follows:
$$V_h=\begin{array}{|c|} \hline
Q_4 \\\hline
Q_4\pm 16 \\\hline
Q_4\pm 32 \\\hline
\vdots \\ \hline
Q_4 \pm 16h   \\ \hline
\end{array}\;,\quad 
W_h=\begin{array}{|c|} \hline
Q_6 \\\hline
Q_4\pm  24\\\hline
Q_4\pm  40\\\hline
\vdots \\ \hline
Q_4 \pm (16h+8)   \\ \hline
\end{array}\;.$$
So, the blocks $V_h$ and $W_h$ have rows and columns that sum to zero. Furthermore,
$\supp(V_h)=[1,16h+16]$ and $\supp(W_h)=[1,16h+24]$ (clearly, $V_0=Q_4$ and $W_0=Q_6$).

If $k\equiv 3 \pmod 4$, write $k=4q+7$  and define
$Z_i=\begin{array}{|c|}\hline U_i \\\hline V_q\pm (3n+(16q+16)i)  \\\hline\end{array}$.
Note that 
$$\bigcup_{i=0}^{\frac{n-4}{4}}\supp\left(V_q\pm (3n+(16q+16)i) \right)=\bigcup_{i=0}^{\frac{n-4}{4}} [3n+1+16(q+1)i, 3n+16(q+1)(i+1) ].$$
If $k\equiv 1 \pmod 4$, write $k=4q+9$  and define
$Z_i=\begin{array}{|c|}\hline U_i \\\hline W_q\pm (3n+(16q+24)i)  \\\hline\end{array}$.
We have
$$\bigcup_{i=0}^{\frac{n-4}{4}} \supp\left( W_q\pm (3n+(16q+24)i) \right)=\bigcup_{i=0}^{\frac{n-4}{4}} [3n+1+8(2q+3)i, 3n+8(2q+3)(i+1) ].$$
In both case, we conclude that $\supp(\mathcal{Z})=[1,3n]\cup [3n+1, kn] =[1,kn]$.

Hence, for all $k\neq 5$ odd, we were able to construct the required sequence $\mathcal{Z}$. 
To obtain an integer $\H(m,n;s,k)$, we  fill the cell $( a+ki, b+4i)$ of an array $H$ of size $m\times n$ with the 
element of the cell $(a,b)$ of $Z_i$, for all $i=0,\ldots,\frac{n-4}{4}$.
Clearly, every column of $H$ contains exactly $k$ filled cells (a column of a unique block $Z_i$).
Also, every row of $H$ contains $s$ filled cells (a row from each $\frac{s}{4}$ distinct blocks $Z_j$),
since $k\frac{n}{4}=m \frac{s}{4}$.
Hence, 
$\supp(H)=\supp(\mathcal{Z})=[1,nk]$ and, since the elements of each row/column of $Z_i$ sum to $0$, the same holds for $H$. 
\end{proof}

Following the proof of Theorem \ref{main_odd}, we can construct an integer $\H(14,8;4,7)$.
We start taking the following two $7\times 4$ blocks:
$$Z_0=\begin{array}{|c|c|c|c|}\hline
4 & -7 & 5 & -2 \\\hline
9 & 18 & -15 & -12 \\\hline
-13 & -11 & 10 & 14 \\ \hline
25 & -26 & -27 & 28 \\\hline
-29 & 30 & 31 & -32 \\\hline
-33 & 34 & 35 & -36 \\ \hline
37 & -38 & -39 & 40 \\\hline
\end{array}\;,\quad 
Z_1=\begin{array}{|c|c|c|c|}\hline
6 &3 & -1 & -8 \\\hline
17 & 19 & -20 & -16 \\\hline
-23 & -22 & 21 & 24 \\\hline
41 & -42 & -43 & 44 \\\hline
-45 & 46 & 47 & -48 \\\hline
-49 & 50 & 51 & -52 \\\hline
53 & -54 & -55 & 56 \\\hline
\end{array}\,;$$
then,
$$\begin{array}{|c|c|c|c|c|c|c|c|}\hline
4 & -7 & 5 & -2 &&&& \\\hline
9 & 18 & -15 & -12 &&&& \\\hline
 -13 & -11 & 10 & 14 &&&& \\ \hline
25 & -26 & -27 & 28 &&&& \\\hline
-29 & 30 & 31 & -32 &&&& \\\hline
-33 & 34 & 35 & -36 &&&& \\ \hline
37 & -38 & -39 & 40 &&&&  \\\hline  
&&&& 6 &3 & -1 & -8 \\\hline
&&&& 17 & 19 & -20 & -16 \\\hline
&&&& -23 & -22 & 21 & 24 \\\hline
&&&& 41 & -42 & -43 & 44 \\\hline
&&&& -45 & 46 & 47 & -48 \\\hline
&&&& -49 & 50 & 51 & -52 \\\hline
&&&& 53 & -54 & -55 & 56 \\\hline
  \end{array}
 $$
is an integer $\H(14,8;4,7)$. An integer $\H(28,16; 12,21)$ is shown in Figure \ref{fig1}, where we highlighted the blocks
$U_0,U_1,U_2,U_3$.

\begin{figure}
\rotatebox{90}{
\begin{footnotesize}
$\begin{array}{|c|c|c|c|c|c|c|c|c|c|c|c|c|c|c|c|c|c|c|c|}\hline
\omb 8 &\omb -15 &\omb 9 &\omb -2 & -129 & 143 & 130 & -144 & -237 & 238 & 239 & -240 &  &  &  & \\ \hline
\omb 17 &\omb  36 & \omb -29 &\omb  -24 & 131 & -141 & -132 & 142 & -241 & 242 & 243 & -244 &  &  &  &\\ \hline
\omb -25 & \omb -21 &\omb  20 & \omb 26 & 145 & -146 & -147 & 148 & 245 & -246 & -247 & 248 &  &  &  & \\\hline
49 & -62 & -51 & 64 & -149 & 150 & 151 & -152 & 249 & -250 & -251 & 252 &  &  &  & \\\hline
-50 & 61 & 52 & -63 & -153 & 154 & 155 & -156 & -253 & 254 & 255 & -256 &  &  &  & \\\hline
53 & -66 & -55 & 68 & 157 & -158 & -159 & 160 & -257 & 258 & 259 & -260 &  &  &  & \\\hline
-54 & 65 & 56 & -67 & 161 & -162 & -163 & 164 & 261 & -262 & -263 & 264 &  &  &  & \\\hline
-57 & 71 & 58 & -72 & -165 & 166 & 167 & -168 &  &  &  &  & \omb 5 & \omb -14 & \omb -3 & \omb 12\\\hline
59 & -69 & -60 & 70 & -169 & 170 & 171 & -172 &  &  &  &  & \omb 38 & \omb -33 & \omb -39 & \omb 34\\\hline
73 & -74 & -75 & 76 & 173 & -174 & -175 & 176 &  &  &  &  & \omb -43 &\omb  47 & \omb 42 & \omb -46\\\hline
-77 & 78 & 79 & -80 & 177 & -178 & -179 & 180 &  &  &  &  & 265 & -278 & -267 & 280\\\hline
-81 & 82 & 83 & -84 & -181 & 182 & 183 & -184 &  &  &  &  & -266 & 277 & 268 & -279\\\hline
85 & -86 & -87 & 88 & -185 & 186 & 187 & -188 &  &  &  &  & 269 & -282 & -271 & 284\\\hline
89 & -90 & -91 & 92 & 189 & -190 & -191 & 192 &  &  &  &  & -270 & 281 & 272 & -283\\\hline
-93 & 94 & 95 & -96 &  &  &  &  & \omb 18 &\omb  28 & \omb -19 &\omb  -27 & -273 & 287 & 274 & -288\\\hline
-97 & 98 & 99 & -100 &  &  &  &  & \omb 13 &\omb  -6 & \omb -11 & \omb 4 & 275 & -285 & -276 & 286\\\hline
101 & -102 & -103 & 104 &  &  &  &  &\omb  -31 &\omb  -22 & \omb 30 & \omb 23 & 289 & -290 & -291 & 292\\\hline
105 & -106 & -107 & 108 &  &  &  &  & 193 & -206 & -195 & 208 & -293 & 294 & 295 & -296\\\hline
-109 & 110 & 111 & -112 &  &  &  &  & -194 & 205 & 196 & -207 & -297 & 298 & 299 & -300\\\hline
-113 & 114 & 115 & -116 &  &  &  &  & 197 & -210 & -199 & 212 & 301 & -302 & -303 & 304\\\hline
117 & -118 & -119 & 120 &  &  &  &  & -198 & 209 & 200 & -211 & 305 & -306 & -307 & 308\\\hline
 &  &  &  & \omb 10 & \omb 7 & \omb -1 & \omb -16 & -201 & 215 & 202 & -216 & -309 & 310 & 311 & -312\\\hline
 &  &  &  & \omb 35 & \omb 37 & \omb -40 & \omb -32 & 203 & -213 & -204 & 214 & -313 & 314 & 315 & -316\\\hline
 &  &  &  & \omb -45 &\omb  -44 & \omb 41 &\omb  48 & 217 & -218 & -219 & 220 & 317 & -318 & -319 & 320\\\hline
 &  &  &  & 121 & -134 & -123 & 136 & -221 & 222 & 223 & -224 & 321 & -322 & -323 & 324\\\hline
 &  &  &  & -122 & 133 & 124 & -135 & -225 & 226 & 227 & -228 & -325 & 326 & 327 & -328\\\hline
 &  &  &  & 125 & -138 & -127 & 140 & 229 & -230 & -231 & 232 & -329 & 330 & 331 & -332\\\hline
 &  &  &  & -126 & 137 & 128 & -139 & h233 & -234 & -235 & 236 & 333 & -334 & -335 & 336\\\hline
\end{array}$
\end{footnotesize}}
\caption{An integer  $\H(28,16; 12,21)$.}\label{fig1}
\end{figure}

\section{Signed magic arrays and magic rectangles}\label{sec:magic}

The existence of  magic rectangles and signed magic arrays has been considered mainly in the square case, in the tight case,
or when each column contains two or three filled cells, see \cite{H1,H2,magic2,rect3,rect1,magic3,magicDM}.
We can provide a reduction theorem also for  magic rectangles and signed magic arrays. 
In fact,  easily adapting the proof of Theorem \ref{sq->rt}, we obtain the following result, where 
we denote a magic square $\MR(n,n;k,k)$ and a signed magic square $\SMA(n,n;k,k)$ by $\MR(n;k)$ and $\SMA(n;k)$, respectively.
Furthermore, we say that an $\MR(n;k)$ (or an $\SMA(n;k)$) is diagonal if it is cyclically $k$-diagonal.

\begin{thm}\label{magic sq->rt}
Let $m,n,s,k$ be four integers such that $2\leq s \leq n$, $2\leq k\leq m$ and $ms=nk$.
Set $d=\gcd(s,k)$.
If  there exists a (shiftable) diagonal $\SMA\left(\frac{nk}{d};d \right)$,
then there exists a (shiftable)  $\SMA( m,n; s,k)$.
If  there exists a diagonal $\MR\left(\frac{nk}{d}; d \right)$,
then there exists an  $\MR( m,n; s,k)$.
\end{thm}

Now, we consider the existence of diagonal signed magic arrays, starting with the following constructions.

\begin{lem}\label{3diag_odd}
There exists a diagonal $\SMA(a;3)$ for all odd $a\geq 3$. 
\end{lem}

\begin{proof}
A diagonal $\SMA(a;3)$, when $a=2g+1$, can be obtained by taking a pf array $A$
whose rows $R_1,R_2,\ldots,R_a$ are as follows.
For each $\ell \in [0,g]$, we fill the row $R_{1+2\ell}$ with the entries
$$\E(1+2\ell,1+2\ell) =-(2g+1)+\ell, \quad \E(1+2\ell, 2+2\ell) = g-2\ell, \quad  
\E(1+2\ell, 3+2\ell) = (g+1)+\ell,$$
and, for each $\ell \in [0,g-1]$, we fill the row $R_{2+2\ell}$ with the entries
$$\E(2+2\ell,2+2\ell)=-(3g+1)+\ell,\quad \E(2+2\ell,3+2\ell)=(g-1)-2\ell,\quad
\E(2+2\ell,4+2\ell)= (2g+2)+\ell.$$
It is clear that we filled the diagonals $D_{0},D_{1},D_{2}$ in such a way that the elements of every row of $A$ sum to zero.
Furthermore, $\E(D_{0})= [-(3g+1),-(g+1)]$,  $\E(D_1)=[-g,g]$ and  $\E(D_2)=[g+1, 3g+1 ]$.
We now describe the columns $C_1,C_2,\ldots,C_a$  of $A$. 
For each $\ell \in [0,g]$, we filled the column $C_{2+2\ell}$ with the entries
$$\E(2\ell, 2+2\ell)= (2g+1)+\ell, \quad  \E(1+2\ell, 2+2\ell) =g-2\ell, \quad 
\E(2+2\ell, 2+2\ell)= -(3g+1)+\ell.$$
Note that the column $C_1$ corresponds to the value $\ell=g$, since the indexes are taken modulo $a$. Also,
for each $\ell\in [0,g-1]$, we filled the column $C_{3+2\ell}$ with the entries
$$\E(1+2\ell, 3+2\ell)=(g+1)+\ell, \quad \E(2+2\ell, 3+2\ell)= (g-1)-2\ell,\quad
\E(3+2\ell, 3+2\ell)= -2g+\ell.$$
Thus, the elements of every column of $A$ sum to zero: we conclude that $A$ is a diagonal $\SMA(a;3)$.
\end{proof}

\begin{lem}\label{3diag_even}
There exists a diagonal $\SMA(a;3)$ for all $a\geq 6$ such that $a\equiv 2\pmod 4$. 
\end{lem}

\begin{proof}
A diagonal $\SMA(a;3)$, when $a=4g+2$, can be obtained by taking a pf array $A$
whose rows $R_1,R_2,\ldots,R_a$ are as follows.
We fill the row $R_1$ with the entries
$$ \E(1,1) =-(3g+2), \quad \E(1,2)=1,\quad \E(1,3)= 3g+1.$$
For each $\ell \in [0,g-1]$, we fill the row $R_{2+2\ell}$ with the entries
$$\E(2+2\ell, 2+2\ell)= -(5g+4) -\ell, \quad 
\E(2+2\ell, 3+2\ell)=  3+2\ell, \quad \E(2+2\ell, 4+2\ell)= (5g+1)-\ell,$$
and, for each $\ell \in [0,g-2]$, we fill the row $R_{3+2\ell}$ with the entries
$$\E(3+2\ell, 3+2\ell)= -(3g+4)-\ell, \quad 
\E(3+2\ell, 4+2\ell)= 4+2\ell, \quad \E(3+2\ell, 5+2\ell )= 3g-\ell.$$
Next, we fill the row $R_{2g+1}$ with the entries
$$\E(2g+1, 2g+1)=-(4g+3),\quad \E(2g+1,2g+2)=2, \quad \E(2g+1,2g+3)=4g+1.$$
For each $\ell \in [0,g-1]$, we fill the row $R_{2g+2+2\ell}$ with the entries
$$\E(2g+2+2\ell, 2g+2+2\ell)=-(4g+4) - \ell, $$ 
$$\E(2g+2+2\ell, 2g+3+2\ell)=- (2g-1) +2\ell, \quad \E(2g+2+2\ell, 2g+4+2\ell)= (6g+3) -\ell,$$
and, for each $\ell \in [0,g-2]$, we fill the row $R_{2g+3+2\ell}$ with the entries
$$\E(2g+3+2\ell,2g+3+2\ell)= -(2g+2) - \ell, $$
$$\E(2g+3+2\ell, 2g+4+2\ell)= -(2g-2) +2\ell,\quad \E(2g+3+2\ell, 2g+5+2\ell)=  4g -\ell.$$
Finally, we fill the rows $R_{a-1},R_a$ with the entries
$$\E(a-1,a-1)=-(3g+1),\quad \E(a-1,a)=-(2g+1), \quad \E(a-1,1)=5g+2,$$
$$\E(a,a)= -(3g+3),\quad \E(a,1)=-2g, \quad \E(a,2)=5g+3.$$
It is clear that we filled the diagonals $D_{0},D_{1},D_{2}$ in such a way that the elements of every row of $A$ sum to zero.
Furthermore, $\E(D_{0})=[-(6g+3),-(2g+2)] $,  $\E(D_1)=[-(2g+1),-1]\cup [1,2g+1]$ and  $\E(D_2)=[2g+2,6g+3]$.
We now describe the columns $C_1,C_2,\ldots,C_a$  of $A$. 
We filled the columns $C_1,C_2$ with the entries
$$\E(a-1,1)=5g+2,\quad \E(a,1)=-2g,\quad \E(1,1)=-(3g+2),$$
$$\E(a,2)=5g+3,\quad \E(1,2)=1, \quad  \E(2,2)=-(5g+4).$$
For each $\ell \in [0,g-1]$, we filled the column $C_{3+2\ell}$ with the entries
$$\E(1+2\ell, 3+2\ell)=(3g+1)-\ell,$$
$$\E(2+2\ell, 3+2\ell)=3+2\ell, \quad \E(3+2\ell, 3+2\ell)=-(3g+4)-\ell,$$
and, for each $\ell\in [0,g-2]$, we filled the column $C_{4+2\ell}$ with the entries
$$\E(2+2\ell, 4+2\ell)=(5g+1)-\ell,$$
$$\E(3+2\ell, 4+2\ell)=4+2\ell, \quad \E(4+2\ell, 4+2\ell)=-(5g+5)-\ell.$$
Next, we filled the column $C_{2g+2}$ with the entries
$$\E(2g,2g+2)=4g+2, \quad \E(2g+1,2g+2)=2,\quad \E(2g+2,2g+2)=-(4g+4).$$
Then, for each $\ell \in [0,g-1]$, we filled the column $C_{2g+3+2\ell}$ with the entries
$$\E(2g+1+2\ell, 2g+3+2\ell)=(4g+1)-\ell$$
$$\E(2g+2+2\ell, 2g+3+2\ell)=-(2g-1)+2\ell, \quad
\E(2g+3+2\ell, 2g+3+2\ell)=-(2g+2)-\ell,$$
and, for each $\ell\in [0,g-2]$, we filled the column $C_{2g+4+2\ell}$ with the entries
$$\E(2g+2+2\ell, 2g+4+2\ell)=(6g+3)-\ell,$$
$$\E(2g+3+2\ell, 2g+4+2\ell)=-(2g-2)+2\ell, \quad
\E(2g+4+2\ell, 2g+4+2\ell)=-(4g+5 )-\ell.$$
Finally, we filled the column $C_{a}$ with the entries
$$\E(a-2,a)=5g+4,\quad \E(a-1,a)=-(2g+1),\quad \E(a,a)=-(3g+3).$$
Thus, the elements of every column of $A$ sum to zero: we conclude that $A$ is a diagonal $\SMA(a;3)$.
 \end{proof}

In Figure \ref{fig5} the reader can find a diagonal $\SMA(14;3)$, obtained by following the proof of Lemma \ref{3diag_even}.

\begin{figure}[ht]
\begin{footnotesize}
$$\begin{array}{|c|c|c|c|c|c|c|c|c|c|c|c|c|c|}\hline
     -11 &    1 &   10 &    &   &   &   &   &   &   &   &   &   &   \\\hline
      & -19 &    3 &   16 &    &   &   &   &   &   &   &   &   &   \\\hline
      &   & -13 &    4 &    9 &    &   &   &   &   &   &   &   &    \\\hline
      &   &   & -20 &    5 &   15 &    &   &   &   &   &   &   &    \\\hline
     &   &   &   & -14 &    6 &    8 &    &   &   &   &   &   &   \\\hline
      &   &   &   &   & -21 &    7 &   14 &    &   &   &   &   &   \\\hline
      &   &   &   &   &   & -15 &    2 &   13 &    &   &   &   &   \\\hline
      &   &   &   &   &   &   & -16 &   -5 &   21 &    &   &   &    \\\hline
      &   &   &   &   &   &   &   &  -8 &   -4 &   12 &    &   &    \\\hline
      &   &   &   &   &   &   &   &   & -17 &   -3 &   20 &    &    \\\hline
      &   &   &   &   &   &   &   &   &   &  -9 &   -2 &   11 &     \\\hline
      &   &   &   &   &   &   &   &   &   &   & -18 &   -1 &   19  \\\hline
     17 &    &   &   &   &   &   &   &   &   &   &   & -10 &   -7  \\\hline
     -6 &   18 &    &   &   &   &   &   &   &   &   &   &   & -12  \\\hline
  \end{array}$$
\end{footnotesize}
\caption{A diagonal $\SMA(14;3)$.}\label{fig5}
\end{figure}

\begin{lem}\label{5diag_odd}
There exists a diagonal $\SMA(a;5)$ for all odd $a\geq 5$. 
\end{lem}

\begin{proof}
A diagonal $\SMA(a;5)$, when  $a=2g+1$, can be obtained by taking a pf array $A$
whose rows $R_1,R_2,\ldots,R_a$ are as follows.
We fill the row $R_1$ with the entries
$$\E(1,1)=-(5g+2),\quad \E(1,2)= -(g+1),\quad \E(1,3)=g,\quad \E(1,4)= g+2,\quad \E(1,5)=4g+1.$$
For each $\ell \in [0,g-1]$, we fill the row $R_{2+2\ell}$ with the entries
 $$\E(2+2\ell, 2+2\ell)= -(4g+2)-\ell, \quad 
\E(2+2\ell, 3+2\ell)= -(3g+1)+2\ell, $$
$$\quad \E(2+2\ell,  4+2\ell)= (g-1)-2\ell,\quad \E(2+2\ell, 5+2\ell)= (g+3)+2\ell, \quad \E(2+2\ell, 6+2\ell)= (5g+1)-\ell,$$
and, for each $\ell \in [0,g-2]$, we fill the row $R_{3+2\ell}$ with the entries
$$\E(3+2\ell, 3+2\ell) = -(3g+2)-\ell, \quad
\E(3+2\ell, 4+2\ell)=  -3g +2\ell,$$
$$\E(3+2\ell, 5+2\ell)= (g-2)-2\ell,\quad \E(3+2\ell, 6+2\ell)= (g+4)+2\ell, \quad \E(3+2\ell, 7+2\ell)= 4g -\ell.$$
Finally, we fill the row $R_a$ with the entries
$$\E(a,a)= -(4g+1),\quad \E(a,1)= -(g+2),\quad \E(a,2)= -g, $$
$$\E(a,3)=g+1, \quad \E(a,4)= 5g+2.$$
It is clear that we filled the diagonals $D_{0},D_{1},D_{2},D_3,D_4$ in such a way that the elements of every row of $A$ sum to zero.
Furthermore, $\E(D_0)=[-(5g+2), -(3g+2)]$, $\E(D_1)=[-(3g+1), -(g+1)]$,  $\E(D_2)=[-g,g]$, $\E(D_3)=[g+1,3g+1]$
and $\E(D_4)=[3g+2, 5g+2]$.
We now describe the columns $C_1,C_2,\ldots,C_a$  of $A$. 
For each $\ell \in [0,g]$, we filled the column $C_{3+2\ell}$ with the entries
$$\E(-1+2\ell, 3+2\ell)= (4g+2)-\ell, \quad \E(2\ell, 3+2\ell) = (g+1) + 2\ell,$$
$$\E(1+2\ell, 3+2\ell)= g -2\ell, \quad \E(2+2\ell, 3+2\ell)= -(3g+1)+2\ell,$$
$$\E( 3+2\ell,  3+2\ell)= -(3g+2) -\ell.$$
Note that the column $C_2$ corresponds to the value $\ell=g$. Also,  for each $\ell\in [0,g-1]$, we filled the column $C_{4+2\ell}$ with the entries
$$\E(2\ell, 4+2\ell)= (5g+2)-\ell, \quad \E( 1+2\ell, 4+2\ell)= (g+2) +2\ell,$$
$$\E(2+2\ell,  4+2\ell)= (g-1)-2\ell, \quad \E( 3+2\ell, 4+2\ell)= -3 g +2\ell, \quad
\E( 4+2\ell,  4+2\ell )= -(4g+3)-\ell.$$
Note that the column $C_1$ corresponds to the value $\ell=g-1$. 
Thus, the elements of every column of $A$ sum to zero: we conclude that $A$ is a diagonal $\SMA(a;5)$.
\end{proof}

\begin{lem}\label{6diag}
There exists a shiftable diagonal $\SMA(a;6)$ for all $a\geq 6$. 
\end{lem}

\begin{proof}
For any $a\geq 6$, a shiftable diagonal $\SMA(a;6)$ can be obtained by taking a pf array $A$
whose rows $R_1,R_2,\ldots,R_a$ are as follows.
For each $\ell \in [1,a-3]$, we fill the row $R_{\ell}$ with the entries
 $$\E(\ell,\ell)= \ell, \quad \E(\ell,\ell+1)= a+\ell, \quad \E(\ell, \ell+2)=-a-2\ell,$$
$$\E(\ell,\ell+3)= -3-\ell,\quad \E(\ell, \ell+4)= (3a-2)-\ell, \quad \E(\ell, \ell+5)= -(3a-5)+2\ell.$$
We fill the rows $R_{a-2},R_{a-1}$ as follows:
$$\E(a-2,a-2)=a-2,\quad \E(a-2,a-1)=2a-2,\quad \E(a-2,a)=-(3a-4),$$
$$\E(a-2,1)=-1, \quad \E(a-2,2)=3a, \quad \E(a-2,3)=-(3a-1);$$
$$\E(a-1,a-1)=a-1, \quad \E(a-1,a)=2a-1,\quad \E(a-1,1)=-(3a-2),$$
$$\E(a-1,2)=-2,\quad \E(a-1,3)=3a-1,\quad  \E(a-1,4)=-(3a-3).$$
Finally, we fill the row $R_a$ with the entries
$$\E(a,a)=a,\quad \E(a,1)=2a,\quad  \E(a,2)=-3a,$$
$$\E(a,3)=-3,\quad \E(a,4)=3a-2,\quad \E(a,5)=-(3a-5).$$

It is clear that we filled the diagonals $D_{0},D_{1},\ldots,D_5$ in such a way that the elements of every row of $A$ sum to zero.
Note that each row contains three positive entries and three negative entries.
Also,  $\E(D_0)=[1,a]$, $\E(D_1)=[a+1,2a]$,  $\E(D_2)=\{-3a,-(3a-2),\ldots,-(a+2)\}$, $\E(D_3)=[-a,-1]$,
$\E(D_4)=[2a+1,3a]$ and  $\E(D_5)=\{-(3a-1), -(3a-3),\ldots,-(a+1)\}$.
We now describe the columns $C_1,C_2,\ldots,C_a$  of $A$. 
We filled the columns $C_1,C_2$ with the entries
$$\E(a-4,1)=-(a+3),\quad \E(a-3,1)=2a+1,\quad \E(a-2,1)=-1,$$
$$\E(a-1,1)=-(3a-2),\quad \E(a,1)=2a,\quad \E(1,1)=1;$$
$$\E(a-3,2)=-(a+1),\quad \E(a-2,2)=3a,\quad \E(a-1,2)=-2,$$
$$\E(a,2)=-3a,\quad \E(1,2)=a+1,\quad \E(2,2)=2.$$
For each $\ell \in [3,a]$, we filled the column $C_{\ell}$ with the entries
$$\E(\ell-5,\ell)=-(3a+5)+2\ell,\quad \E(\ell-4,\ell)= (3a+2)-\ell,\quad \E(\ell-3,\ell)=-\ell,$$
$$\E(\ell-2,\ell)=-(a-4)-2\ell, \quad \E(\ell-1,\ell)=(a-1)+\ell, \quad \E(\ell,\ell)=\ell.$$
Thus, the elements of every column of $A$ sum to zero, and every column contains three positive entries and three negative entries.
We conclude that $A$ is a shiftable diagonal $\SMA(a;6)$.
\end{proof}

Note that, because of Lemmas \ref{3diag_odd}, \ref{3diag_even}, \ref{5diag_odd} and \ref{6diag},
many of the constructions given in \cite{magicDM} can be shortened and simplified.

\begin{cor}\label{magic_diag}
There exists a diagonal $\SMA(a;b)$ for any $a\geq b\geq 3$.
Also, if $b$ is even, there exists a shiftable diagonal $\SMA(a;b)$ for any $a\geq b\geq 3$.
\end{cor}

\begin{proof}
The statement follows from \cite[Theorems 7 and 9]{magicDM} when 
(i) $b$ is even and $a$ is odd; or (ii) $b\geq 5$ is odd and $a$ is even; or (iii) $b=3$ and $a\equiv 0 \pmod 4$.
The case $b=3$ and $a\equiv 2 \pmod 4$ follows from Lemma \ref{3diag_even}, while
the case $b\equiv 0\pmod 4$ and $a$ even follows from \cite[Lemmas 10 and 11]{magicDM}.

We are left with the following two cases: (1) $a,b$ are both odd; (2) $a$ is even and $b\equiv 2 \pmod 4$.
So, write $b=4h+r$ with $r\in\{3,5,6\}$. 
Let $A$ be the  diagonal $\SMA(a;r)$ constructed in Lemmas \ref{3diag_odd}, \ref{5diag_odd} and \ref{6diag},
and let $B$ be the shiftable diagonal $\SMA(a;4)$ of \cite[Lemma 10]{magicDM}.
If $h=0$, set $C=A$. If $h>0$, let $C$ be the pf array obtained by 
filling the $4h$ diagonals $D_{r},D_{r+1},\ldots,D_{r+4h-1}$ of $A$ with the elements of the filled diagonals of 
$B\pm \frac{ar-1}{2}, B\pm \frac{a(r+4)-1}{2},\ldots, B\pm \frac{a(r+4h-4)-1}{2}$.
Then, $C$ is a diagonal $\SMA(a;b)$.
\end{proof}

\begin{prop}\label{shiftSMA}
Let  $m,n,s,k$ be four integers such that $3\leq s\leq n$, $3\leq k \leq m$ and $ms=nk$.
If $s$ and $k$ are both even, there  exists a shiftable $\SMA(m,n;s,k)$.
\end{prop}
 
 \begin{proof}
This result was already proved in \cite{MP2}, except when $s,k\equiv 2 \pmod 4$ and $m,n$ are both odd.
In this case, without loss of generality, we may assume $m\geq n$ (and so $s\leq k$).
Let $A_1$ be a shiftable $\SMA(n;s)$ obtained in Corollary \ref{magic_diag}.
Clearly if $m=n$ we have nothing to prove. So, suppose $m>n$.
Since $m-n\geq 2$ is even and $k-s\equiv 0 \pmod 4$ with $k-s\geq 4$, by
\cite[Proposition 4.14]{MP2} there exists a shiftable $\SMA(m-n, n;s,k-s)$, say $A_2$.
Let $A$ be the pf array of size $m\times n$ obtained by taking
$$A=\begin{array}{|c|}\hline
   A_1\\\hline
   A_2\pm ns/2 \\\hline
   \end{array}.$$
Each row of  $A$ contains $s$ filled cells and each of its columns contains $s+(k-s)=k$ filled cells.
Also, note that $\E(A_1)=\{\pm 1, \pm 2, \ldots, \pm ns/2\}$ and
$\E(A_2\pm sn/2)=\{\pm (1+ns/2), \pm (2+ns/2), \ldots, \pm ms/2\}$.
Since  $\E(A)=\E(A_1)\cup \E(A_2)=\{\pm 1, \pm 2, \ldots, \pm ms/2\}$,  $A$ is a shiftable $\SMA(m,n;s,k)$.
 \end{proof}

\begin{proof}[Proof of Theorem \ref{magic}]
Set $d=\gcd(s,k)$. If $d=2$, then $s$ and $k$ are both even, and the result follows from Proposition \ref{shiftSMA}.
Assume $d\geq 3$. By Corollary \ref{magic_diag}, there exists a diagonal $\SMA\left(\frac{nk}{d};d\right)$:
the result follows from Theorem \ref{magic sq->rt}. 
\end{proof}

We now consider the existence of magic rectangles.

\begin{lem}\label{magic->rect}
Let  $m,n,s,k$ be four integers such that $2\leq s\leq n$, $2\leq k \leq m$ and $ms=nk$.
If there exists a shiftable $\SMA(m,n;s,k)$, then there exists an $\MR(m,n;s,k)$. 
\end{lem}

\begin{proof}
See the proof of  \cite[Theorem 1.12]{MP2}.
\end{proof}

Thanks to the classification of the diagonal $\MR(a;b)$ given in \cite{rect1}, we get Theorem \ref{magic_rect}.

\begin{proof}[Proof of Theorem \ref{magic_rect}]
If $d$ is even, the result follows from Proposition  \ref{shiftSMA} and Lemma \ref{magic->rect}.
In \cite{rect1} it was proved that a diagonal $\MR(a;b)$ exists
for all $a\geq b\geq 3$ such that $a$ is odd or $b$ is even.
So, if $d\geq 3$ and $nk$ is odd, apply  Theorem \ref{magic sq->rt}, taking  $b=d$ and $a=\frac{nk}{d}$.
\end{proof}

In Figure \ref{fig3}, we give an $\MR(9,18;12,6)$ obtained by applying our reduction theorem to  a diagonal $\MR(18;6)$,
whose construction is described in the proof of \cite[Theorem 6]{rect1}.

\begin{figure}[ht]
\begin{footnotesize}
$$\begin{array}{|c|c|c|c|c|c|c|c|c|}\hline
0 & 34 & 36 & 70 & 106 & 75 &  &  & \\\hline
 & 1 & 33 & 37 & 69 & 104 & 77 &  & \\\hline
 &  & 2 & 32 & 38 & 68 & 102 & 79 & \\\hline
 &  &  & 3 & 31 & 39 & 67 & 100 & 81\\\hline
83 &  &  &  & 4 & 30 & 40 & 66 & 98\\\hline
96 & 85 &  &  &  & 5 & 29 & 41 & 65\\\hline
64 & 94 & 87 &  &  &  & 6 & 28 & 42\\\hline
43 & 63 & 92 & 89 &  &  &  & 7 & 27\\\hline
26 & 44 & 62 & 90 & 91 &  &  &  & 8\\\hline
9 & 25 & 45 & 61 & 88 & 93 &  &  & \\\hline
 & 10 & 24 & 46 & 60 & 86 & 95 &  & \\\hline
 &  & 11 & 23 & 47 & 59 & 84 & 97 & \\\hline
 &  &  & 12 & 22 & 48 & 58 & 82 & 99\\\hline
101 &  &  &  & 13 & 21 & 49 & 57 & 80\\\hline
78 & 103 &  &  &  & 14 & 20 & 50 & 56\\\hline
55 & 76 & 105 &  &  &  & 15 & 19 & 51\\\hline
52 & 54 & 74 & 107 &  &  &  & 16 & 18\\\hline
35 & 53 & 71 & 72 & 73 &  &  &  & 17\\\hline
  \end{array}$$ 
  \end{footnotesize}
 \caption{An $\MR(9,18;12,6)$.}\label{fig3}
\end{figure}

To conclude our paper, we show how the ideas of Section \ref{s0} can be adapted to construct signed magic arrays.
Also in this case, we prefer to work in the more general context of integer $\lambda$-fold relative Heffter arrays.

\begin{prop}
Let $m,n,s,k$ be four integers such that $2\leq s\leq n$, $2\leq k \leq m$ and $ms=nk$.
Let $\lambda$ be a divisor of $n$ such that $n\geq 3\lambda$. Suppose that $n,s,\lambda$ are even and 
such that $\frac{n}{\lambda}k\equiv 0,3\pmod 4$.
Then, there exists an integer ${}^\lambda\H_1(m,n;s,k)$.
\end{prop}

\begin{proof}
Let $H$ be an integer $\H\left(k,\frac{n}{\lambda};\frac{n}{\lambda},k\right)$, and call $C_1,C_2,\ldots,C_{\frac{n}{\lambda}}$ its columns.
By \cite{ABD}, $H$ exists since $\frac{n}{\lambda}\geq 3$ and $\frac{n}{\lambda}k\equiv 0,3\pmod 4$.
For any $i=1,\ldots,\frac{n}{\lambda}$, define the $k\times 2$ block
$F_i=\begin{array}{|c|c|}\hline
C_{i} & -C_{i}\\\hline
     \end{array}$.
Hence, each block $F_i$ has rows and columns that sum to zero.     
Now, take the sequence $\B$ obtained by concatenating $\frac{\lambda}{2}$ copies of the sequence $(F_1,F_2,\ldots,F_{\frac{n}{\lambda}})$.
The sequence $\B$  has length $\frac{n}{2}$ and is such that each of the elements in $[1,k \frac{n}{\lambda}]$ appears in 
$\E(\B)$ exactly $\frac{\lambda}{2}$ times with sign $+$ and $\frac{\lambda}{2}$ times with sign $-$.
Write $\B=(B_0,B_1,\ldots,B_{\frac{n-2}{2}})$.
To obtain an integer ${}^\lambda\H_1(m,n;s,k)$, we  fill the cell $(a+ki, b+2i)$ of an array $A$ of size $m\times n$ with the 
element of the cell $(a,b)$ of block $B_i$, for all $i=0,\ldots,\frac{n-2}{2}$.
Also in this case, every column of $A$ contains exactly $k$ filled cells (a column of a unique block $B_i$),
and every row of $A$ contains $s$ filled cells (a row from each $\frac{s}{2}$ distinct blocks $B_j$),
since $k\frac{n}{2}=m \frac{s}{2}$.
\end{proof}

In general, a ${}^2\H_1(m,n;s,k)$ is not necessarily a signed magic array.
However, by the particular shape of the pf array constructed in the proof of the previous proposition,
for $\lambda=2$ we get the following.

\begin{cor}\label{ultimo}
Let $m,n,s,k$ be four integers such that $3\leq s\leq n$, $3\leq k \leq m$ and $ms=nk$.
If $n$ and $s$ are even and $nk\equiv 0,6\pmod 8$, then there exists an $\SMA(m,n;s,k)$.
\end{cor}

In Figure \ref{fig4} we give an $\SMA(20,8;6,15)$ obtained by following the previous construction.

\begin{figure}[ht]
\begin{footnotesize}
$$\begin{array}{|c|c|c|c|c|c|c|c|}\hline
   7 & -7 & -34 & 34 & -32 & 32 &  & \\\hline
-12 & 12 & 33 & -33 & -47 & 47 &  & \\\hline
-2 & 2 & -38 & 38 & 48 & -48 &  & \\\hline
6 & -6 & 37 & -37 & 51 & -51 &  & \\\hline
1 & -1 & 43 & -43 & -52 & 52 &  & \\\hline
21 & -21 & -41 & 41 &  &  & 19 & -19\\\hline
-22 & 22 & -54 & 54 &  &  & -20 & 20\\\hline
25 & -25 & 53 & -53 &  &  & 11 & -11\\\hline
-26 & 26 & 58 & -58 &  &  & -14 & 14\\\hline
-29 & 29 & -57 & 57 &  &  & 4 & -4\\\hline
31 & -31 &  &  & -10 & 10 & 36 & -36\\\hline
45 & -45 &  &  & 17 & -17 & -35 & 35\\\hline
-46 & 46 &  &  & -18 & 18 & 40 & -40\\\hline
-49 & 49 &  &  & 3 & -3 & -39 & 39\\\hline
50 & -50 &  &  & 8 & -8 & -44 & 44\\\hline
 &  & -16 & 16 & -23 & 23 & 42 & -42\\\hline
 &  & 15 & -15 & 24 & -24 & 56 & -56\\\hline
 &  & 9 & -9 & -27 & 27 & -55 & 55\\\hline
 &  & 5 & -5 & 28 & -28 & -60 & 60\\\hline
 &  & -13 & 13 & 30 & -30 & 59 & -59\\\hline
  \end{array}$$
\end{footnotesize}
 \caption{An $\SMA(20,8;6,15)$.}\label{fig4}
\end{figure}

Now, we show how to construct an $\SMA(m,n;s,k)$ when $k$ is odd and $s\equiv 0 \pmod 4$.
It will be useful to introduce the following notation: if $0\leq a<b $, then $\pm [a,b]$ denotes the set $[-b,-a]\cup [a,b]$.
As done in Section \ref{s0}, we begin considering the case $k=3$. 
Working with \emph{signed} Skolem sequences \cite[Theorem 3.1]{BJ}, we produce $3\times 4$ 
blocks whose rows and columns sum to zero. In fact, for any fixed integer $\mu\geq 0$, we can define:
$$\begin{array}{rcl}
A & =& \begin{array}{|c|c|c|c|}\hline
1 &  2 &  -2 &  -1 \\ \hline
 4\mu+3 &  -(4\mu+6) &  -(12\mu+3) & 12\mu+6 \\\hline
 -(4\mu+4) &  4\mu+4 &  12\mu+5 &   -(12\mu+5) \\ \hline
\end{array}\;,\\[14pt]
B_a   & =& \begin{array}{|c|c|c|c|}\hline
 4a+3 &  4a+5 &  -(4a+5) &  -(4a+3) \\ \hline
  4\mu+4a+5 &  -(4\mu+8a+12) &  -(12\mu-8a-3) &  12\mu-4a+4  \\ \hline
 -(4\mu+8a+8) &  4\mu+4a+7 &  12\mu-4a+2 &  -(12\mu-8a+1)  \\ \hline
  \end{array}\;,\\[14pt]
C_a & =& \begin{array}{|c|c|c|c|}\hline
 4a+4 & 4a+6 & -(4a+6)& -(4a+4)  \\ \hline
 4\mu+4a+6 &  -(4\mu+8a+14) & -(12\mu-8a-5) &  12\mu-4a+3 \\ \hline
 -(4\mu+8a+10) & 4\mu+4a+8 &  12\mu-4a+1 &  -(12\mu-8a-1) \\ \hline
\end{array}\;, 
 \end{array}$$
with $a\in [0,\mu-1]$. The sequence
\begin{equation}\label{mseqU1}
\mathcal{A}(\mu) = \left\{
\begin{array}{ll}
(A) & \textrm{if } \mu=0,\\
(A,B_0,B_1,\ldots,B_{\mu-1}, C_0, C_1,\ldots,C_{\mu-1})&  \textrm{if } \mu\geq 1
\end{array}\right.
\end{equation}
has length $2\mu+1$ and $\E(\mathcal{A}(\mu))=\pm [1,12\mu+6]$.
In fact, 
$$\begin{array}{rcl}
\E(A) & =& \{-(12\mu+5),-(12\mu+3)\}\cup \{-(4\mu+6), -(4\mu+4)\}\cup [-2,-1]\cup [1,2]\\
&&\cup [4\mu+3,4\mu+4]\cup [12\mu+5,12\mu+6]; \\
\bigcup\limits_{a=0}^{\mu-1} \E(B_a) & = & \{-(12\mu+4), -12\mu,\ldots, -(4\mu+8) \}\cup \{-(12\mu+1),-(12\mu-3),\ldots,\\
&&-(4\mu+5)\}\cup \{-(4\mu+1), -(4\mu-1), \ldots,-3\}\cup \{3,5,\ldots,4\mu+1\} \\
&& \cup \{4\mu+5, 4\mu+7, \ldots, 8\mu+3\} \cup \{8\mu+6,8\mu+8,\ldots,12\mu+4\};\\
\end{array}$$
$$\begin{array}{rcl}
\bigcup\limits_{a=0}^{\mu-1} \E(C_a) & = & \{-(12\mu+6), -(12\mu+2), \ldots,-(4\mu+10)\}\cup 
\{-(12\mu-1), -(12\mu-5), \\
&&\ldots, -(4\mu+3) \} \cup
\{-(4\mu+2), -4\mu,\ldots,-4\} \cup  \{4,6,\ldots,4\mu+2\}\cup \\
&&\{4\mu+6,4\mu+8,\ldots,8\mu+4\}\cup \{8\mu+5,8\mu+7,\ldots,12\mu+3\}.
\end{array}$$

\begin{proof}[Proof of Theorem \ref{magic4}]
If $k$ is even, the statement follows from Theorem \ref{magic}. 
So, assume $k$ odd: from the hypothesis $ms=nk$ we obtain that $n\equiv 0 \pmod 4$.
Furthermore, if $n\equiv 0 \pmod 8$, the statement follows from Corollary \ref{ultimo}.
So, we may also assume $n\equiv 4 \pmod 8$.

Our first step is to construct a sequence $\mathcal{Z}=(Z_0,Z_1,\ldots,Z_{\frac{n-4}{4}})$ of length $\frac{n}{4}$,
such that its elements $Z_i$ are $k\times 4$ blocks whose rows and columns sum to zero, and such that $\E(\mathcal{Z})=
\pm [1,\frac{nk}{2}]$.
To this purpose, take the sequence $\B=\mathcal{A}\left(\frac{n-4}{8}\right)$ and write
$\B=(U_0,U_1,\ldots,U_{\frac{n-4}{4}})$. Note that  $\E(\mathcal{B})=\pm  [1,\frac{3n}{2}]$.
If $k=3$, set $Z_i=U_i$ for all $i=0,\ldots,\frac{n-4}{4}$. It is clear that $\mathcal{Z}$ satisfies all our requirements.
Assume $k\geq 5$. To obtain our block $Z_i$ we use the block $U_i$ for the first three rows, and we need a shiftable 
$\SMA(k-3,4;4,k-3)$
to construct the remaining $k-3$ rows. So, consider the following shiftable block:
$$Q=\begin{array}{|c|c|c|c|}\hline
1 & -2 & -3 & 4 \\ \hline
-1 & 2 & 3 & -4 \\ \hline
    \end{array}.$$
For all fixed $h\geq 0$ we take the $(2h+2)\times 4$ block $V_h$:
$$V_h=\begin{array}{|c|} \hline
Q \\\hline
Q\pm 4 \\\hline
Q\pm 8 \\\hline
\vdots \\ \hline
Q \pm 4h   \\ \hline
\end{array}.$$
The block $V_h$  has rows and columns that sum to zero; furthermore, $\E(V_h)=\pm  [1,4h+4]$.
Write $k=2q+5$  and define
$Z_i=\begin{array}{|c|}\hline U_i \\\hline V_q\pm (\frac{3n}{2}+(4q+4)i)  \\\hline\end{array}$.
Note that 
$$\begin{array}{rcl}
\bigcup\limits_{i=0}^{\frac{n-4}{4}}\E\left(V_q\pm (\frac{3n}{2}+(4q+4)i) \right) & =& 
\bigcup\limits_{i=0}^{\frac{n-4}{4}} \pm[\frac{3n}{2}+1+4(q+1)i, \frac{3n}{2}+4(q+1)(i+1) ]\\
& =& \pm[\frac{3n}{2}+1,\frac{3n}{2}+(q+1)n]=\pm[\frac{3n}{2}+1,\frac{nk}{2}].
     \end{array}$$
We conclude that $\E(\mathcal{Z})=\pm [1,\frac{nk}{2}]$.

Hence, for all $k$ odd, we were able to construct the required sequence $\mathcal{Z}$. 
To obtain an $\SMA(m,n;s,k)$, once again, we  fill the cell $( a+ki, b+4i)$ of an array $H$ of size $m\times n$ with the 
element of the cell $(a,b)$ of $Z_i$, for all $i=0,\ldots,\frac{n-4}{4}$.
Clearly, every column of $H$ contains exactly $k$ filled cells (a column of a unique block $Z_i$),
and every row of $H$ contains $s$ filled cells (a row from each $\frac{s}{4}$ distinct blocks $Z_j$),
as $k\frac{n}{4}=m \frac{s}{4}$.
Hence, 
$\E(H)=\E(\mathcal{Z})=\pm[1,\frac{nk}{2}]$ and, since the elements of each row/column of $Z_i$ sum to $0$,
the same holds for $H$. 
\end{proof}

\end{document}